\documentclass[11pt]{article}
\usepackage{amsmath,amsthm,amssymb,a4wide}
\usepackage[twoside]{geometry}
\usepackage{graphicx}
\usepackage{epsfig}    
   
\newcommand \Prob {\text{Prob}}

\newcommand \loc {\text{loc}} 

\newtheorem{structure}{Assumption}

\newtheorem{theorem}{Theorem}[section]
\newtheorem{proposition}[theorem]{Proposition}
\newtheorem{lemma}[theorem]{Lemma}
\newtheorem{corollary}[theorem]{Corollary}
\newtheorem{definition}[theorem]{Definition}
\newtheorem{remark}[theorem]{Remark} 
\numberwithin{equation}{section} 
\numberwithin{figure}{section}

\newcommand \dVbound {d\overline V}

 \newcommand \Tcal {\mathcal T}

\newcommand \widehatg {\widehat g}
\newcommand \widehatt {\widehat t}

\newcommand \hatu {\widehat u}
\newcommand \hatv {\widehat v}

\newcommand \Mcal {\mathcal M} 
 
\newcommand \trianglerightNEW \triangleright

\newcommand \delh {\widehat \del} 
\newcommand \uh {\widehat u} 
\newcommand \vh {\widehat v}

\newcommand \qt {\widetilde q}

\newcommand \auth {\textsc}

\newcommand \bei {\begin{itemize}}
\newcommand\eei {\end{itemize}}
\newcommand \be {\begin{equation}}
\newcommand \bel {\begin{equation}\label}
\newcommand\ee {\end{equation}}
\newcommand \del \partial
\newcommand \RR {\mathbb R}

\newcommand \eps \epsilon 
\let\oldmarginpar\marginpar
\renewcommand\marginpar[1]{\-\oldmarginpar[\raggedleft\footnotesize #1]%
{\raggedright\footnotesize #1}}
  
\begin{document}

\title{\bf \Large The finite volume method on a Schwarzschild background}

\author{Shijie Dong$^*$ and Philippe G. LeFloch\footnote{
\normalsize Laboratoire Jacques-Louis Lions, Centre National de la Recherche Scientifique, Sorbonne Universit\'e, 
4 Place Jussieu, 75252 Paris, France. 
\newline
Email : {\sl dongs@ljll.math.upmc.fr, contact@philippelefloch.org}
}}

\date{January 2019}

\maketitle

\abstract
We introduce a class of nonlinear hyperbolic conservation laws on a Schwarz\-schild black hole background and derive several properties satisfied by (possibly weak) solutions. Next, we formulate a numerical approximation scheme which is based on the finite volume methodology and takes the curved geometry into account. An interesting feature of our model is that no boundary conditions is required at the black hole horizon boundary. We establish that this scheme converges to an entropy weak solution to the initial value problem and, in turn, our analysis also provides us with a theory of existence and stability for a new class of conservation laws. 
\endabstract

\vfill

\setcounter{tocdepth}{1}
\tableofcontents

\vfill 
  \newpage
 

\section{Introduction} 
\label{section1}

We design and study a finite volume scheme for a class of nonlinear hyperbolic equations posed on a Schwarzschild black hole background. This paper is the follow-up of earlier investigations by LeFloch and co-authors \cite{BL,LM,LX1,LX2}. As is common in the mathematical theory of hyperbolic balance laws, we consider a (drastically) simplified version of the compressible Euler equations and we describe the fluid evolution by a single scalar unknown function, typically representing the velocity of the fluid. 
For relativistic problems the velocity is naturally bounded and, after normalization, we seek for solutions 
\be
v : \Mcal \to [-1, 1]
\ee
defined on a ``spacetime'' $\Mcal$ ---explicitly described below in a global coordinate chart\footnote{so that the present paper is aimed at a reader interested in the discretization of nonlinear hyperbolic equation with variable coefficients.}--- 
and satisfying the following hyperbolic balance law  
\bel{eq:geometric1}
\nabla_\alpha \big( X^\alpha(v, \cdot) \big) = q(v, \cdot) \quad \text{ in } \Mcal.
\ee 
Here, $X^\alpha=X^\alpha(w, \cdot)$ is the so-called flux vector field parametrized by the real variable 
$w \in [-1,1]$ and defined on $\Mcal$, while $q=q(w, \cdot)$ is a prescribed real-valued function. 
Structural conditions (even for smooth solutions, as specified later in this text) must be imposed on the vector field in order for the balance law to admit a well-posed initial value formulation.  

Our objectives in this paper are as follows: 

\bei 

\item Choosing $\Mcal$ to be (the outer domain of communication of) a Schwarzschild black hole, we introduce a class of hyperbolic balance laws \eqref{eq:geometric1} and formulate the associated initial value problem. We then seek for weak solutions $v: \Mcal \to [-1, 1]$ possibly containing shock waves which must satisfy a suitable entropy condition (discussed below). 

\item Next, we design a finite volume scheme that allows us to numerically approximate these weak solutions and we derive several fundamental properties of interest: maximum principle, entropy inequalities, etc. 
We establish the strong convergence of this scheme toward a weak solution of the initial value problem. 

\eei 

\noindent Our arguments are based on a generalization of DiPerna's theory of measure-valued solutions \cite{DiPerna}
and require us to cope with the effects of the curved black hole geometry.  

An outline of this paper is as follows. 
In Section \ref{sec:formulation} we introduce the class of hyperbolic equations of interest and provide a motivation from pressureless fluid dynamics. In Section \ref{sec:characteristic}, we analyze the geometry of the curved characteristics in the black hole geometry and the class of steady state solutions which represent a fluid at rest. In Section \ref{sec:interior}, we discuss an alternative choice of slicing and which illustrate how the balance gets transform under change of coordinates. In Section \ref{sec:FVM}, we introduce our finite volume scheme and state the convergence theory. The entropy inequalities satisfied by the weak solutions and their discrete version are also derived, and the proof of convergence is completed.


\section{Formulation based on the Schwarzschild coordinates}
\label{sec:formulation}

\subsection{The choice of coordinates}

\begin{subequations}
The domain of outer communication of a Schwarzschild black hole, 
denoted by $\Mcal$, can be described in the so-called Schwarzschild coordinates $x=(t,x^j)= (t,x^1, \ldots, x^n)$ in which the spacetime metric reads 
\bel{eq:101} 
g = - \Big(1 - {2 M \over r} \Big) \, dt^2 + \Big(1 - {2 M \over r} \Big)^{-1} dr^2 + r^2 \, g_{S^{n-1}}. 
\ee 
Here, the time variable $t$ and the radius $r$ defined by $r^2 := \sum_{j=1}^n (x^j)^2$ satisfy 
\be
t \in [0, + \infty), \qquad r \in (2M, +\infty). 
\ee
The light speed is normalized to unit while the parameter 
$M \in [0, +\infty)$
represents the mass of the black hole. 
Moreover, $g_{S^2}$ denotes the canonical metric on the unit $(n-1)$-sphere $S^{n-1} \subset \RR^n$.
The spacetime hypersurface 
\be
\big\{ r=2M \big\} \subset \overline \Mcal 
\ee
is the boundary of our spacetime and 
represents the horizon of the black hole, from which nothing can propagate in the (outer communication) domain $r>2M$ of interest. Recall that the apparent singularity at $r=2M$ in the expression of the metric \eqref{eq:101} is not a physical singularity but is solely due to our choice of coordinates. 

\end{subequations}

\begin{remark}
Passing to the so-called Eddington-Finkelstein coordinates would allow us to eliminate this singularity, but at the expense of adding further complexity in the algebraic expressions. Fortunately, the coordinates in \eqref{eq:101}  are suitable for our purpose of analyzing the dynamics of a fluid outside the horizon. 
See Section \ref{sec:interior} for a different choice of coordinates. 
\end{remark}


\subsection{The model of interest}

Choosing the vector field in the left-hand side of the balance law \eqref{eq:geometric1} to be 
$$
X= \Big({1 \over \sqrt{\det (g)}}
{v \over (1- {2 M \over r})^2}, {1 \over \sqrt{\det (g)}} {f(v)\over {1- {2 M \over r}}}, 0, \ldots, 0 \Big)
$$ 
and the source term to be
$$
q(v, x)= {2 M \over r^2 \Big( 1 - {2 M \over r} \Big)^2} h(v), 
$$
we arrive at the following hyperbolic balance law:  
\bel{eq:Rmodel}
\del_t \Bigg({v \over (1- {2M \over r})^2}\Bigg) + \del_r \Bigg({f(v)\over 1- {2M \over r}} \Bigg)
= {2 M \over r^2 \Big(1 - {2 M \over r}\Big)^2} h(v). 
\ee
Here, the functions $f=f(w)$ and $h=h(w)$ are prescribed functions, while the unknown scalar field  is $v: \RR_+ \times \Omega \mapsto [-1,1]$, defined for all $t \geq 0$ and $r \geq 2M$, and we work in the exterior of the ball with radius $2M$, that is 
\be
\Omega := \big \{ r >2M \big\} \subset \RR^n. 
\ee

In our model the unknown $v$ need not be spatially symmetric, so it convenient to rewrite \eqref{eq:Rmodel} in Cartesian coordinates, i.e. 
\bel{eq:Cmodel}
\del_t \Bigg({v \over \Big(1- {2M \over r} \Big)^2}\Bigg)
+ \del_j \Bigg( {x^j \over r \Big(1- {2M \over r}\Big)} f(v)  \Bigg) 
- {(n-1) \over r \Big(1- {2M \over r} \Big)} f(v) 
= {2 M  \over r^2 \Big(1 - {2 M \over r} \Big)^2} \,  h(v).
\ee
Finally, in order to eliminate the singularity ${1\over 1- {2M /  r}}$, we propose an equivalent form, as follows. 

\begin{definition} The equation with unknown $v: \RR_+ \times \Omega \mapsto [-1,1]$
\bel{eq:Nmodel}
\aligned
& \del_t v+ \del_j \Big( \Big(1-{2M \over r} \Big) {x^j \over r} f(v) \Big) = g(v, r), 
\\
& g(v, r) := 
\del_j \Big( \Big(1- {2M \over r} \Big) {x^j \over r} \Big) f(v) 
+ {2 M \over r^2} \big( f(v) +  h(v) \big) 
\endaligned
\ee
is referred to as a {\bf hyperbolic balance law on a Schwarzschild black hole}. 
\end{definition}

 
At this juncture, it should be emphasized that further conditions (presented in Section \ref{sec:characteristic})
will be required on the flux function $f$ in order for the interval $[-1, 1]$ to be an invariant domain. 

\begin{definition} A pair of functions $(U, F): [-1,1] \to \RR \times \RR$ is called a {\bf convex entropy pair} for the equation \eqref{eq:Nmodel} if the function $v \in [-1,1] \mapsto U(v)$ is convex and 
\bel{eq:entropy} 
F' (v) =   f'(v) U'(v), \qquad v \in [-1,1]. 
\ee
\end{definition} 

We always tacitly assume that an entropy $U$ is normalize to satisfy $U(0) = 0$.
Then, by definition an entropy solution to the equation \eqref{eq:Nmodel} must satisfy, for all convex entropy pair $(U, F)$, 
\be 
\del_t U(v) + \Big(1-{2M \over r} \Big) {x^j \over r} \del_j F (v)
\leq U'(v) {2 M \over r^2} \Big(f(v) +h(v)\Big).
\ee
We prescribe an initial data $v_0$ at the time $t=0$, that is, 
\bel{eq:initial} 
v(0, \cdot) = v_0
\ee
and we  formalize our notion of solution as follows. 

\begin{definition}\label{def-entropy} 
Given a measurable function $v_0: \Omega \to [-1,1]$, a
 measurable function $v: \RR_+ \times \Omega \mapsto [-1,1]$ is called an {\bf entropy solution} to the Cauchy problem \eqref{eq:Nmodel} and \eqref{eq:initial} if the following inequalities hold
\bel{entropy-ineq2} 
\aligned
& \int_0^{+ \infty} \int_\Omega 
\Big( U(v) \del_t \phi + F(v) \del_j \big( \Big(1-{2M \over r} \Big) {x^j \over r} \phi \big) 
+ U'(v) {2 M \over r^2} \Big(f(v) +h(v)\Big)  \phi \Big) \, dxdt
\\
& 
 + \int_\Omega U(v_0) \phi(0, \cdot) \, dx \geq 0
\endaligned
\ee 
for all convex entropy pairs $(U, F)$ and all compactly supported test-functions $\phi \geq 0$. 
\end{definition}
 

\subsection{Derivation from the relativistic Euler system} 
\label{sub:moti}

Our motivation for introducing the above class of balance laws comes from a formal derivation 
made from the Euler equations for a relativistic compressible fluid, which read 
\bel{eq:Euler}
\nabla_\alpha \big( T^{\alpha\beta}(\rho,u) \big) = 0,
\ee
in which $\nabla$ denotes the Levi-Civita connection associated with the Schwarzschild metric \eqref{eq:101}. We are interested here in the energy-momentum tensor of a pressureless fluid, given by 
\bel{eq:tensor-form}
T^{\alpha\beta} (\rho,u) = \rho u^\alpha u^\beta,
\ee
where $\rho:  M \mapsto (0, +\infty)$ denotes the density of the fluid and  the velocity field $u= (u^\alpha)$ is normalized to be future-oriented, unit and timelike $u^\alpha u_\alpha = g_{\beta\beta'} u^\alpha u^{\beta'} =-1$ with $u^0 > 0$ and, therefore, 
\bel{eq:normalized} 
-1= -\Big(1 - {2 M \over r} \Big) (u^0)^2+ \Big(1 - {2 M \over r} \Big)^{-1} (u^1)^2. 
\ee
By assuming spherical symmetry, we can derive from the above system a single equation satisfied by a suitably normalized component of the velocity field, denoted below by $v \in (-1,1)$. 

As usual, by taking
\bel{eq:scalar}
v:= {1 \over {(1 - 2M/ r)}}{ {u^1 \over u^0}},
\ee
we get 
\be 
(u^0)^2= {1 \over {(1 - v^2) (1 - 2 M/r)}}, 
\qquad
(u^1)^2= {v^2 \over {1- v^2}} (1- 2 M/ r).
\ee
Elementary computations (following \cite{LX1}) yield us 
$$
\aligned
 \del_t \Big({ \rho \over {1- v^2}}\Big) + \Big({1- {2M \over r}} \Big) \del_r  \Big({\rho v \over {1-v^2}}\Big) + \rho{{v(2r- 2M)} \over {r^2(1- v^2)}} & = 0, 
\\
 \del_t \Big({\rho v \over {1- v^2}}\Big) + \Big({1- {2M \over r}}\Big) \del_r \Big({\rho v^2 \over {1-v^2}}\Big) + \rho{{M(1- 3v^2) +2v^2 r} \over {r^2(1- v^2)}}
& = 0.
\endaligned
$$
Combining these two equations together, we get 
\bel{eq:relatiBurgers}
\del_t \Big({v \over (1- {2M \over r})^2}\Big) + \del_r \Big({{v^2/2}\over{1- {2M \over r}}}\Big) + {M\over{r^2(1- {2M \over r})^2}}=0.
\ee

We now compare \eqref{eq:relatiBurgers} with  \eqref{eq:Cmodel}. Restricting now attention to radially symmetric solutions, then \eqref{eq:Cmodel} is equivalent to
 \bel{eq:103}
\del_t \Big({v \over (1- 2M /r)^2}\Big) + \del_r \Big({f(v) \over 1- 2M/r} \Big)= {2 M  r^2 \over (1 - 2 M/r)^2} h(v).
\ee
Clearly, this latter equation includes \eqref{eq:relatiBurgers} as a special case, obtained by taking 
\bel{eq:105}
f(s)= s^2/2- 1/2, \qquad h(s)=0. 
\ee
Hence, we can regard \eqref{eq:Cmodel} as a generalization to \eqref{eq:relatiBurgers}.


\section{Characteristics and steady states}
\label{sec:characteristic}

\subsection{Maximum principle} 

The method of characteristics allows us to obtain a first insight about the properties of (sufficiently regular solutions) to our balance law \eqref{eq:Cmodel}. It leads to ordinary differential equations along characteristic curves parametrized with respect to some time parameter (denoted by $s$ below). 
We would like to deduce some properties of solution $v$ by proposing the following assumption on the flux $f$ and the source $h$. 

\begin{structure}
\label{atr:assumption1}
The flux and source functions are assumed to satisfy 
\bel{eq:assum-flux}
f(\pm 1) + h(\pm1 )=0, \qquad f'(\pm 1) + h'(\pm 1) \neq 0.
\ee 
\end{structure}
 
We motivate our condition by the following analysis along characteristic curves. So, we consider the coupled system 
\bel{eq:ode-set} 
\aligned
{dt \over ds} & = {1 \over ({1- {2M \over r}})^2},
\\
{dx^j \over ds} & = {x^j \over {r( {1- {2M \over r})}}}f' \big(u(s) \big),
\\
u(s) & = v(t(s), x(s)). 
\endaligned
\ee
A straightforward computation shows that 
\bel{eq:112}
\aligned
u'(s) 
& = \del_t v {dt \over ds} + \del_j v {dx^j \over ds}
  = \del_t v {1 \over ({1- {2M \over r}})^2} + \del_j v {x^j \over {r( {1- {2M \over r})}}}f' \big(u(s) \big)
\\
& = \Bigg({{d-1}\over {r( {1- {2M \over r})}}}- {\del_j \Big( {x^j \over {r( {1- {2M \over r})}}}} \Big)\Bigg) f \big(u(s) \big) + {2M \over (r- 2M)^2} \, h \big(u(s) \big)
\\
& = {2M \over (r- 2M)^2} \Big( f \big(u(s) \big) + h \big(u(s) \big) \Big).
\endaligned       
\ee
This equation tells us how the values of a solution evolve along characteristics, and we use it in order to establish a maximum principle. It is convenient to assume a strict inequality in the data. 

\begin{proposition}[Maximum principle] \label{prop:max}
Consider the balance law \eqref{eq:Nmodel} under the condition \eqref{eq:assum-flux}. Then, given any initial data \eqref{eq:initial} satisfying
$$
\sup_\Omega | v_0 |  \leq 1,
$$
the solution $v=v(t,x)$ satisfies the same bound for all times 
$$
\sup_\Omega | v(t, \cdot) | \leq 1, \qquad t \geq 0, 
$$
as long as they remain sufficiently regular. 
\end{proposition}

\begin{proof}
Observe first that if $v_0 = \pm 1$ initially then it remains so for all times. 
It is sufficient to show that $v \leq 1$, since exactly the same arguments apply to showing $v \geq -1$.

Consider first the case of $H'(1) > 0$ with $H(s):= f(s) +h(s)$. By continuity, we have $H(s) < 0$ for all $s \in (1-\eps, 1)$ and some $\eps$. Hence, if $u \in [1-\eps, 1)$, \eqref{eq:112} implies that $u'<0$ and, consequently, $u\leq 1$ for all times. 

In the case $H'(1)<0$, we have $H(s) < 0$ for all $s \in (1, 1+ \eps)$ and some $\eps$. Recall $\sup | v_0 | < 1$, if $u(s_2)>1$ for some $s_2$, then $u(s)>1$ for $s\in (s_1, s_2)$ with $s_1 := \sup \{ s: u(s)\leq 1, s < s_2 \}$. We see that $u(s_1)=1$. However, if we integrate \eqref{eq:112} in $[s_1, s_2]$, the right-hand side would be negative ,while the 
left-hand side 
would be positive. Hence, $u \leq 1$ (actually, if $u$ could reach 1 at some $s$, then it must remain identically $1$ afterwards).
\end{proof} 


\subsection{Geometry of the characteristic curves}

\begin{subequations}
Along a characteristic we see that
\bel{eq:drds} 
{d r \over d s} 
= 
\del_j r {d x^j \over d s} 
= 
{f'(u) \over 1 - 2 M / r}.
\ee
Recalling \eqref{eq:112}, we get in the $(v, r)$ plane
\bel{eq:odeur} 
{d u \over d r} 
= 
{2 M \big(h(u) + f(u)\big) \over (r - 2 M) r f'(u)}
\ee
\end{subequations}
and, more explicitly,
\bel{eq:odeurS1} 
\widehat{F}(u) - \widehat{F}(u_0)  
= 
\log \Big({1 - {2 M / r} \over 1 - {2 M / r_0}} \Big), 
\qquad 
\widehat{F}(u)
:= 
\int_0^u {f'(w) \over h(w) + f(w)} \, d w, 
\ee
where $r_0 = r(s_0)$ and $u_0 = u(s_0)$ are given data at some time $s_0$.

Concerning the global behavior of the characteristics in the special case $f(w)= w^2/2 - 1/2$ and $h(w)=0$, which is the Burgers equation posed on the Schwarzschild background, the weak solutions in the $(u, r)$ plane can be expressed in terms of the initial data via a minimisation formulation based on characteristics; see \cite{BL}.


Here, to proceed with the study of the characteristic curves and for the sake of definitness, we assume some specific signs about the functions $f$ and $h$.  

\begin{structure} 
\label{atr:assumption2}
\bel{eq:fh} 
f(s) + h(s) < 0, \qquad s \in (-1, 1)
\ee
and
\bel{eq:f'} 
\aligned
& f'(s) < 0, \quad s \in (-1, 0),
\\
& f'(s) > 0, \quad s \in (0, 1).
\endaligned
\ee
\end{structure}

A direct consequence from \eqref{eq:fh} and \eqref{eq:f'} is that
\begin{eqnarray}
\widehat{F}(w)   \text{ is }
\left \{
\begin{array}{lll}
& \text{increasing and negative}, \qquad & w \in (-1, 0),
\\
& 0, \qquad & w = 0,
\\
& \text{decreasing and negative}, \qquad & w \in (0, 1).
\end{array}
\right.
\end{eqnarray}
We rewrite \eqref{eq:odeurS1} as
\bel{eq:odeurS2} 
\widehat{F}(u) 
= 
\log \Big(e^{\widehat{F}(u_0)}{1 - {2 M / r} \over 1 - {2 M / r_0}} \Big)
\ee
and, by solving for $u$, the ordinary differential equation \eqref{eq:drds} for the radius function $r(s)$ can be written as  
\begin{eqnarray}
{dr \over ds} 
= 
\left \{
\begin{array}{lll}
& \Big( 1 - {2M \over r} \Big)^{-1} f' \Bigg( \widehat{F}_+^{(-1)} \bigg( \log \Big(e^{\widehat{F}(u_0)}{1 - {2 M / r} \over 1 - {2 M / r_0}} \Big) \bigg) \Bigg) \quad \text{when } e^{\widehat{F}(u_0)}{1 - {2 M / r} \over 1 - {2 M / r_0}} \geq 1,
\\
& \Big( 1 - {2M \over r} \Big)^{-1} f' \Bigg( \widehat{F}_-^{(-1)} \bigg( \log \Big(e^{\widehat{F}(u_0)}{1 - {2 M / r} \over 1 - {2 M / r_0}} \Big) \bigg) \Bigg) \quad \text{when } e^{\widehat{F}(u_0)}{1 - {2 M / r} \over 1 - {2 M / r_0}} \leq 1. 
\end{array}
\right.
\end{eqnarray}
Here, $\widehat{F}_+^{(-1)}$ and $\widehat{F}_-^{(-1)}$ are the inverse functions of $\widehat{F}_+$ and $\widehat{F}_-$, respectively, and 
\begin{eqnarray}
\widehat{F} (w) 
= 
\left \{
\begin{array}{lll}
&\widehat{F}_-(w), \qquad w\in (-1, 0],
\\
&\widehat{F}_+(w), \qquad w\in [0, 1).
\end{array}
\right.
\end{eqnarray}
Note that $\widehat{F}_{\pm}$ are single-valued functions within the domain of interest.

We follow \cite{BL} and introduce the \emph{escape velocity} (whenever it exists) 
\be 
u_0^E 
:= 
\widehat{F}_+^{(-1)} \bigg( \log \Big( 1 - {2M \over r_0} \Big) \bigg),
\ee
which satisfies the property
\be 
\lim_{r \to +\infty} \Big( 1 - {2M \over r} \Big)^{-1} f' \Big( \widehat{F}_{\pm}^{(-1)} \bigg( \log \Big(e^{\widehat{F}(u_0)}{1 - {2 M / r} \over 1 - {2 M / r_0}} \Big) \bigg) \Big) 
= 
0.
\ee
Replacing the radius $r_0$ by the escape velocity parameter $u_0^E$ in \eqref{eq:odeurS2}, we obtain
\be 
u 
= 
\widehat{F}_{\pm}^{(-1)} \Big( \log \Big( e^{\widehat{F}(u_0) - \widehat{F}(u_0^E)} \big( 1 - 2M/r \big) \Big) \Big).
\ee

\begin{subequations}

The late-time behavior of $u=u(s)$ can be checked to be described as follows: 
\bei

\item \textbf{Negative initial data.} The function 
$u(s)$ decreases (as follows from \eqref{eq:112} and the assumption \eqref{eq:fh}). If $u_0 \in (-1, 0]$ with initial data $(s_0, r_0)$, then $u(s)$ remains negative and decreasing and $r(s)$ decreases towards $ 2 M$. More precisely, we have
\be 
\lim_{s\to +\infty} r(s) 
= 
2M, 
\qquad 
\lim_{s\to +\infty} u(s) 
= 
-1.
\ee

\item \textbf{Positive initial data with $0<u_0 < u_0^E$.} 
The positivity of $u_0$ initially ensures $dr/ds > 0$, that is, the characteristic curve initially moves away from the black hole. However, $u(s)$ keeps decreasing and eventually reaches $0$ at some time $s_0$. The dynamics then coincides with that for negative initial data. We conclude that
\be 
\lim_{s\to +\infty} r(s) 
= 
2M, 
\qquad 
\lim_{s\to +\infty} u(s) 
= 
-1.
\ee

\item \textbf{Positive initial data with $u_0 \geq u_0^E$.} In this case, the characteristic curve moves away from the black hole for all times and the asymptotic behavior is 
\be 
\lim_{s\to +\infty} r(s) 
= 
+ \infty, 
\qquad 
\lim_{s\to +\infty} u(s) 
= \widehat{F}_+^{(-1)} \Big( \widehat{F}(u_0) - \widehat{F}(u_0^E) \Big).
\ee
\eei

\end{subequations}


\subsection{Steady state solutions} 

Finally, let us consider solutions that are steady states representing a fluid at rest in the curved black hole geometry. This is a special class of solutions of interest, for instance, in designing (well-balanced) numerical schemes and in finding test cases. 
In view of the radially symmetric form of our equation \eqref{eq:103} (but possibly for non-radially symmetric solutions), for a time-independent solution we obtain 
the ordinary differential equation 
\bel{eq:116}
\del_r \Big({f(u) \over 1- 2M/r}\Big)
= 
{2 M \over r^2 (1 - 2 M/r)^2} h(u) 
\ee 
and, once again, we get the same ordinary differential equation as \eqref{eq:odeur}
\bel{eq:117}
{d u \over d r} 
= 
{2 M \big(h(u) + f(u)\big) \over (r - 2 M) r f'(u)}.
\ee
Under Assumptions \eqref{atr:assumption1} and \eqref{atr:assumption2}, for any given data $(r_0, u_0)$ we can distinguish between two cases: 

\bei

\item \textbf{Negative $u_0$.} Then $u$ is increasing and
$$
\lim_{r\to +\infty} u(r) 
=  
\widehat{F}_-^{(-1)} \Big( \widehat{F}(u_0) - \widehat{F}(u_0^E) \Big).
$$

\item \textbf{Positive $u_0$.} Then $u$ is decreasing and
$$
\lim_{r\to +\infty} u(r) 
=  
\widehat{F}_+^{(-1)} \Big( \widehat{F}(u_0) - \widehat{F}(u_0^E) \Big).
$$
\eei


\section{Coordinates covering the black hole interior}
\label{sec:interior}

\subsection{An alternative choice of time slicing}

In this section, we illustrate the fact that coordinates can be chosen in many different manners. While, for Schwarschild spacetime, this leads to significantly more involved algebraic expressions, such alternative coordinates may allow one to cover a larger region of the spacetime. For definiteness, in this section we take 
$n = 3$.  
Hence, we now introduce a nonlinear hyperbolic equation posed in a larger domain of the Schwarzschild geometry, obtained by ``crossing'' the horizon and we study the interior of the black hole. 
We follow \cite{DB} and introduce the following metric: 
\begin{subequations}
\bel{eq:in-metric}
\widehatg = -{{R - 2 M}\over{R}} d\widehatt^2+ 2 \, {f_1(R) \over R - R_0} \, d\widehatt  dR + \Big( {R \over R - R_0} \Big)^2 (dR^2+ (R - R_0)^2  \, g_{S^2} ),
\ee
in which $\widehatt$ denotes the time variable and $R$ the radial variable with
\be
f_1(R) : = \sqrt{2 r (M- R_0) + R_0 (2 M- R_0)}. 
\ee 
\end{subequations}
Here $R_0 \in (0, M]$ is a parameter that is fixed, and we observe that the above expression is identical to the metric \eqref{eq:101} in the limit $R_0 \to 0$, for which the radial variables $R$ and $r$ would then coincide. 
This new slicing \eqref{eq:in-metric} allows us to go inside of the black hole (when $R_0 > 0$), and we cover the region $\{r : r+ R_0 - 2 M > 0 \}$, within which the metric remains of a definite Lorenztian signature.

In fact, we can transform \eqref{eq:in-metric} (for a restricted domain of the variables, only) into the metric  \eqref{eq:101}, by setting  
\be
\aligned
& \widehatt  = t + h(R), , \qquad 
R : = r + R_0
\\
& {d h \over d R} = {1 \over 1 - 2 M / R} \sqrt{1 - (1 - 2 M / R){R^2 \over r^2}}.
\endaligned
\ee
In the following, it will be convenient to rely on the vector fields
$$
\delh_0 := \del_{\widehatt }, 
\qquad \delh_1 := \del_{R}, 
\qquad \delh_2 := \del_{\theta},
\qquad \delh_3 := \del_{\phi}.
$$


We rewrite \eqref{eq:in-metric} in the matrix form  
$$
(\widehatg_{\alpha \beta})=
\begin{pmatrix}
-{{R - 2 M}\over{r+ R_0}} & {f_1 \over r} &   0 &  0   \\
 {f_1 \over r} & ({R \over r})^2  & 0   &  0  \\
0 &  0  &  R^2  &  0   \\
0 &  0 & 0   & R^2 \sin^2 \theta
\end{pmatrix},
$$
with inverse
$$
(\widehatg^{\alpha \beta})= 
\begin{pmatrix}
 -({R\over r})^2& {f_1 \over r}  &  0 & 0  \\
{f_1 \over r} & {{R - 2 M}\over{R}}  &  0 & 0  \\
0 &  0  & (R)^{-2}  &  0  \\
0 & 0  & 0   & (R)^{-2} \sin^{-2} \theta
\end{pmatrix}.
$$
After a tedious computation, the (non-vanishing) Christoffel symbols 
$\Gamma^{\mu}_{\alpha \beta} = {1\over 2} \widehatg^{\mu \nu}(\delh_{\alpha} \widehatg_{\beta \nu} + \delh_{\beta} \widehatg_{\alpha \nu}- \delh_{\nu} \widehatg_{\alpha \beta})$
are found to be 
\bel{eq:Chr-sym1}
\aligned
\Gamma^{0}_{00}& = {M f_1 \over r(r+ R_0)^2}, 
\hskip4.5cm
&&
\Gamma^{0}_{01}=\Gamma^{0}_{10} =  {M \over r^2},
\\
\Gamma^{0}_{11}& = {(r+ R_0)^2 (M- R_0) \over r^3 f_1} + {R_0^2 (r+ R_0) \over r^3 f_1},
\\
\Gamma^{0}_{22}& = - {{r+ R_0}\over r} f_1,
\hskip4.5cm
&&
\Gamma^{0}_{33}= - {{r+ R_0}\over r} f_1 \sin^2 \theta,
\\
\Gamma^{1}_{00}& =  {M(r+ R_0 - 2 M)\over (r+ R_0)^3},
\hskip3.4cm
&&
\Gamma^{1}_{01}= \Gamma^{1}_{10}= -{M f_1 \over r(r+ R_0)^2},
\\
\Gamma^{1}_{11}& = - {M \over r^2},
\hskip5.5cm
&&
\Gamma^{1}_{22}= -(r+ R_0- 2 M),
\\
\Gamma^{1}_{33}& = -(r+ R_0- 2 M) \sin^2 \theta,
\hskip2.6cm
&&
\Gamma^{2}_{12}= \Gamma^{2}_{21} =\Gamma^{3}_{13}=\Gamma^{3}_{31} =  {1 \over {r+ R_0}},
\\
\Gamma^{2}_{33}& = -\sin \theta \cos \theta,
\hskip4.3cm
&&
\Gamma^{3}_{23}= \Gamma^{3}_{32} = {\cos \theta \over \sin \theta}.
\endaligned
\ee


\subsection{Formulation of the balance law}

We follow the strategy in the previous section and derive our equation from the pressureless Euler system. 
For the (normalized) vector 
$$
\hatu = \big(\hatu ^0, \hatu ^1, 0, 0):=(u^0 + h'(R)\,u^1, u^1, 0, 0 \big),
$$
we find 
\bel{eq:unit1}
\aligned
-1 = \hatu ^\alpha \hatu _\alpha
    & = -{{r+ R_0 - 2 M}\over{r+ R_0}} (\hatu ^0)^2+ 2  {f_1 \over r} \hatu ^0 \hatu ^1 
+ \Big(1 + {R_0 \over r} \Big)^2(\hatu ^1)^2
    \\
    & = -{R - 2 M \over R} (\hatu ^0 - h'(R) \hatu ^1)^2 + {R \over R - 2 M} (\hatu ^1)^2. 
\endaligned
\ee 

\begin{proposition} 
From the pressureless Euler equations, the velocity vector 
\bel{eq:def-vh}
\hatv 
:= 
{1\over {1 - 2 M / R}} {\hatu ^1 \over \hatu ^0 - h'(R) \hatu ^1} 
\ee
satisfies the nonlinear hyperbolic equation
\bel{eq:burg5}
\Big( 1 + h'(R) \hatv (1 - 2M/R) \Big) \delh_{0} \hatv + {\hatv (1 - 2M/R) } \delh_{1} \hatv - {M \over R^2} \hatv^2 + {M \over R^2} 
= 0. 
\ee 
\end{proposition}

\begin{proof} In view of the notation \eqref{eq:def-vh}, we get
\bel{eq:132}
\aligned
\hatu ^0
& = 
{1 \over \sqrt{(1 - \hatv^2)(1 - 2M/R)}} \bigg( 1 + h'(R) \hatv \Big( 1 - {2 M \over R} \Big) \bigg),
\\
\hatu ^1
& = 
{\hatv \over \sqrt{1 - \hatv^2}} \Big( 1 - {2 M \over R} \Big)^{1/2}, 
\endaligned
\ee
which is one representation of $(\hatu^0, \hatu^1)$.
Plugging these expressions into \eqref{eq:tensor-form}, we obtain
\bel{eq:tensors1}
\aligned
T^{00} 
& = 
{1 \over (1 - \hatv^2)(1 - 2M/R)} \Bigg( 1 + h'(R) \hatv \Big( 1 - {2 M \over R} \Big) \Bigg)^2,
\\
T^{01} 
& = 
T^{10} 
= 
{\hatv \over 1 - \hatv^2} \bigg( 1 + h'(R) \hatv \Big( 1 - {2 M \over R} \Big) \bigg),
\\
T^{11} 
& = 
{\hatv^2 \over 1 - \hatv^2} \Big( 1 - {2 M \over R} \Big).
\endaligned
\ee
From \eqref{eq:Euler}, thus 
$\delh_{\alpha} T^{\alpha \beta} + \Gamma^{\alpha}_{\alpha \gamma} T^{\gamma \beta} + \Gamma^{\beta}_{\alpha \gamma}T^{\alpha \gamma}=0$, 
we get
$$
\aligned
0& =\delh_{0} T^{00} + \delh_{1} T^{10} + (2 \Gamma^{0}_{00} + \Gamma^{1}_{10}) T^{00} + (3 \Gamma^{0}_{01} + \Gamma^{1}_{11} + \Gamma^{2}_{21} + \Gamma^{3}_{31}) T^{01} + \Gamma^{0}_{11} T^{11}
\\
& =: \delh_{0} T^{00} + \delh_{1} T^{10} + S_{0,0} T^{00} + S_{0,1} T^{01} + S_{0,2} T^{11},
\endaligned
$$
$$
\aligned
0& = \delh_{0} T^{01} + \delh_{1} T^{11} + \Gamma^{1}_{00} T^{00} + (\Gamma^{0}_{00} + 3\Gamma^1_{01}) T^{01} + (\Gamma^0_{01} + 2\Gamma^1_{11} + \Gamma^2_{21} + \Gamma^3_{31}) T^{11} 
\\
& =:  \delh_{0} T^{01} + \delh_{1} T^{11} + S_{1,0} T^{00} + S_{1,1} T^{01} + S_{1,2} T^{11},
\endaligned
$$
Now we set 
\bel{eq:141}
q:= {T^{01}\over T^{00}}= {T^{11}\over T^{01}}= {\hatv (1 - 2M/R) \over 1 + h'(R) \hatv (1 - 2M/R)},
\ee
and our calculation leads us to 
\bel{eq:burg1}
T^{00} \delh_{0} q+ T^{01} \delh_{1} q+ S_{1,0} T^{00} + (S_{1,1}- S_{0,0}) T^{01} + (S_{1,2}- S_{0,1}) T^{11}- q S_{0,2} T^{11}=0. 
\ee
Finally, further cumbersome calculations give usthe final form \eqref{eq:burg5}.
\end{proof}

\subsection{Characteristics and maximum principle}

As we mentioned in the beginning of this section, the new metric $\widehatg$ coincides with the Schwarzchild metric $g$ when $r$ is replaced by $R=r+R_0$ in \eqref{eq:101}, hence it is not surprising to have the following result. Namely, if we replace $r$ by $R$ throughout Section \ref{sec:formulation}, then Burgers equation \eqref{eq:relatiBurgers} is equivalent to \eqref{eq:burg5}.
This is easy to check  with 
\be 
\del_0 v = \delh_{0} v, \qquad \del_1 v = h'(R) \delh_{0} v + \delh_{1} v.
\ee
From the equation \eqref{eq:relatiBurgers} we have 
\be 
\del_t v + (1- 2M/R) v \del_R v = \Big( 1+ h'(R) \hatv (1 - 2M/R) \Big) \delh_{0} v + (1- 2M/R) v \delh_{1} v,
\ee
which coincides with \eqref{eq:burg5}. We can now restate our previous results in the new coordinates, and we only discuss in detail the new features.

From our equation \eqref{eq:burg5}, the characteristic curves with $\uh(s) := \vh(t(s), r(s))$ are given by
\bel{eq:new-character}
\aligned
{dt \over ds} 
& =  1 + h'(R) \uh(s) (1 - 2M/R),
\\
{dR \over ds} 
& = (1 - 2M/R) \uh(s). 
\endaligned
\ee
Moreover, we have
\bel{eq:uh-ds}
{d \uh \over ds}
= {M \over R^2} (\uh^2 - 1).
\ee
Similarly to what we did in Proposition \ref{prop:max}, we can check the following result. 

\begin{proposition}[Maximum principle]
Consider the equation \eqref{eq:burg5}. If the initial data satisfies
$$
\sup |\vh_0| \leq 1,
$$
then any smooth solution to \eqref{eq:burg5} also satisfies
$$
\sup |\vh| \leq 1.
$$
\end{proposition}

Now, in the $(\uh, r)$--plane let us observe that 
$$
{d \uh \over dR}
= {\uh^2 - 1 \over \uh} {M \over R (R - 2M)},
$$
which follows from \eqref{eq:new-character} and \eqref{eq:uh-ds}. From this, we obtain 
\bel{eq:uh-R0}
{1 - \uh^2(R) \over 1 - {2M \over R}}
= {1 - \uh^2_0 \over 1 - {2M \over R_0}},
\ee
where $(\uh_0, R_0)$ is the initial location. Thus, our previous conclusions concerning the characteristic curves are recovered here. This is of course not surprising, since we are treating the same differential equation expressed in different coordinates. This second formulation however may have some numerical advantage when the horizon, instead of being fixed as it is in the present model, is dynamical. 


\section{Finite volume method and convergence analysis} 
\label{sec:FVM}

\subsection{Formulation of the finite volume scheme}

Having considered the formulation \eqref{eq:Nmodel} (in Section \ref{sec:formulation})
and the formulation \eqref{eq:burg5} (in Section \ref{sec:interior}), we now study the numerical approximation of the general balance law \eqref{eq:geometric1} in a setting that, in principle, may encompass both formulations. For definiteness, we treat the outer domain of communication so that the space variable varies in a half-line and no boundary condition is required at the boundary. 
The hyperbolic model of interest reads 
\bel{eq:MainModel}
\aligned
&
\nabla_\alpha \big( X^\alpha(v, \cdot) \big) = q(v, \cdot) \quad \text{ in } \Mcal,
\\
& v: \Mcal \to [-1,1], 
\endaligned
\ee
in which $\Mcal$ denotes the outer domain of communication of a Schwarzschild black hole with radius $2M$, as we described earlier. Here $X^\alpha = X^\alpha(w, \cdot)$ is a smooth vector field on $\Mcal$, depending upon the real variable $w \in [-1,1]$. An hyperbolicity condition and a condition at the boundary will be made explicit below. 

We are going to formulate a finite volume scheme for the equation \eqref{eq:MainModel} and  establish its convergence by generalizing the technique of proof in \cite{ALO}. In contrast with  this later work, 
the spacelike slices in $\Mcal$ are non-compact and the flux vector $X^\alpha(\bar{v}, x)$ is no longer assumed to be geometry compatible (i.e. $X^\alpha(\bar{v}, x)$ does not satisfy divergence free condition), and at the (horizon) boundary, no boundary data is needed.


In the class of interest in the present paper, the flux vector field satisfies the following property, which implies that  no boundary condition is needed: the spatial components of the vector field $X(\cdot, x)$ vanishes on the boundary, i.e.
\be 
X^a(\cdot, x) = 0 \quad \text{ on the boundary } \del \Mcal. 
\ee 
Following \cite{ALO} we  design a finite volume scheme for \eqref{eq:MainModel} as follows. We introduce a spacetime triangulation $\Tcal_h  =\bigcup_{K \in \Tcal_h }K$ of $\Mcal$ such that 
the boundary $\del K$ of each element $K$ is the union of three possible types of faces:
\bei 

\item A face $e_K^+$ is spacelike and we denote its future--oriented outward unit normal by $n_{K, e_K^+}$.

\item A face $e_K^-$ (in the past of $e_K^+$) is also spacelike, and we denote its past--oriented outward unit normal by $n_{K, e_K^-}$.

\item A vertical face denoted by $e^0$ is timelike, and whose inward unit normal is denoted by $n_{K, e^0}$, and the union of all of such faces is denoted by $\del^0 K := \del K \setminus \{e_K^+, e_K^- \}$. 

\eei

\noindent By definition, for every pair of distinct elements $K, K'\in \Tcal_h $, the intersection $K\cap K'$ is either a common face of $K, K'$ or a submanifold of co-dimension at least $2$. We use $K_e$ to denote the unique neighbor of $K$ sharing the same edge $e$. We denote by $K^\pm$ the neighbors of $K$ which share the same edge $e^\pm$.

By integrating the equation \eqref{eq:MainModel} over an arbitrary element $K\in \Tcal_h$ and applying the divergence theorem, we obtain 
\bel{eq:204-0}
\aligned
&\int_{e_K^+} g (X(v, p), n_{K, e_K^+}(p)) \, \dVbound(p)
+ \int_{e_K^-} g (X(v, p), n_{K, e_K^-}(p)) \, \dVbound(p)
\\
&- \sum_{e^0 \in \del^0K} \int_{e^0} g (X(v, p), n_{K, e^0}(p)) \, \dVbound(p)
=
\int_{K} q(v, p) \, dV(p). 
\endaligned
\ee
Here, $n$ denotes the exterior and unit, normal vector field along the boundary face under consideration, while  
$\dVbound$ is the induced measure element on the boundary. 
Our finite volume scheme is based on the following approximation formulas, in which $e^0$, etc. denotes an edge of $K$: 
\begin{subequations}
\bei
\item \text{Discretization of the main  variable:}
\bel{eq:205-0}
\int_{e_K^\pm} g (X(v, p), n_{K, e_K^\pm}(p)) \, \dVbound(p)
\simeq 
|e_K^\pm| \mu^X_{K, e_K^\pm}(v_K^\pm).
\ee
\item \text{Discretization  of the flux:}
\bel{eq:206-0} 
\int_{e^0} g (X(v, p), n_{K, e^0}(p)) \, \dVbound(p)
\simeq 
|e^0| f_{K,e^0}(v^-_K, v^-_{K_{e^0}}).
\ee
\item \text{Discretization  of the source  term:}
\bel{eq:206a-0}
\int_{K} q(v, p) \, dV(p)
\simeq 
\sum_{e^0\in \del^0 K} |e^0| \mu^X_{K, e^0}(v^-_K) + |K| \qt^A(v_K^-). 
\ee
\eei
\end{subequations}
Here, the numerical flux $f_{K,e}: \RR^2 \to \RR$ is chosen to satisfy the properties of consistency, conservation and monotonicity:
\begin{subequations}
\bei
\item \text{Consistency property:}
\bel{eq:consis0} 
f_{K, e^0}(v, v)
=
{1\over |e^0|} \int_{e^0} g (X(v, p), n_{K, e^0}(p)) \, \dVbound(p),
\qquad
v \in \RR.
\ee

\item \text{Conservation property:}
\bel{eq:conser0} 
f_{K, e^0}(u, v)
=
-f_{K_{e^0}, e^0}(v, u),
\qquad
u, v \in \RR.
\ee

\item \text{Monotonicity property:}
\bel{eq:monoto0} 
\del_u f_{K, e^0}(u, v)
\geq 
0,
\qquad
\del_v f_{K, e^0}(u, v)
\leq 
0,
\qquad
u, v \in \RR.
\ee
\eei
\end{subequations}
Also we have written 
\be
q(v, p) 
= 
\big(\nabla_a X^a(\cdot, p)\big)(v) + \qt(v, p),
\ee
and 
\be
\mu^X_{K, e}(\bar{v})
:=
{1\over |e|} \int_{e} g (X(\bar{v}, p), n_{K, e}(p)) \, dV_{e}.
\ee

Finally, the finite volume approximations are defined by 
\bel{eq:fv-0}
\aligned
|e_K^+| \mu^X_{K^+, e_K^+}(v_K^+)
= 
|e_K^-| \mu^X_{K, e_K^-}(v_K^-) 
& - \sum_{e^0 \in \del^0K}|e^0| f_{K,e^0}(v^-_K, v^-_{K_{e^0}}) 
\\
& - \sum_{e^0\in \del^0 K} |e^0| \mu^X_{K,e^0}(v^-_K) - |K| \qt^A(v_K^-).
\endaligned
\ee


\subsection{Convergence and existence theory}

Based on the geometric formulation of a finite volume method above, we can now proceed with the analysis of our model problem \eqref{eq:Nmodel}. We integrate \eqref{eq:Nmodel} over an element $K$ and by applying the divergence theorem 
\bel{eq:204}
\aligned
\int_K v(t_{n+1}, \cdot)\, dx 
&= 
\int_K v(t_n, \cdot) \, dx - \int_{t_n}^{t_{n+1}} \int_{\del K}{f(v) \over r}\Big( 1 - {2M \over r} \Big)x \cdot n \, \dVbound dt 
\\
& \quad
+ \int_{t_n}^{t_{n+1}} \int_K g(v, r) \, dxdt,
\\
\int_{t_n}^{t_{n+1}}\int_K g(v, r) \, dxdt 
&= 
\int_{t_n}^{t_{n+1}} \int_K \del_j \Bigg( \Big(1- {2M \over r} \Big) {x^j \over r} \Bigg) f(v) \, dxdt 
\\
& \quad
+ \int_{t_n}^{t_{n+1}} \int_{K} {2 M \over r^2} \Bigg( f(v) + h(v) \Bigg) \, dxdt,
\endaligned
\ee
where $n$ denotes the outward unit normal vector. We apply the following approximations: 
\begin{subequations}
\bei

\item \text{Discretization of the main variable:}
\bel{eq:205}
\int_K v(t_{n+1}, \cdot) \, dx 
\simeq 
|K| v^{n+1}_K, 
\ee

\item \text{Discretization of the flux:}
\bel{eq:206} 
\int_{t_n}^{t_{n+1}}\int_{e}{{f(v) \over r}\Big( 1 - {2M \over r} \Big)x \cdot n_{K, e}} \, \dVbound dt 
\simeq 
\tau |e| f_{K,e}(v^n_K, v^n_{K_e})\omega_{K, e},
\ee
where, with $x_e$ being the center of $e$ and $r_e = |x_e|$, 
\be
\omega_{K, e} := {1 \over r_e } \Big( 1 - {2M \over r_e} \Big) x_e \cdot n_{K, e}. 
\ee

\item \text{Discretization of the source term:}
\bel{eq:206a}
\int_{t_n}^{t_{n+1}}\int_{K} g(v, r) \, dxdt 
\simeq 
\tau \sum_{e\in \del K} |e| f(v^n_K) \omega_{K, e} + \tau |K| \big( f(v^n_K) + h(v^n_K) \big) \theta_K,
\ee
where, $r_K$ being the radial variable evaluated at the center of $K$, 
\be
\theta_K := {2M \over r_K^2}, \qquad 
\ee

\eei
\end{subequations}

In the above, we denoted by $e$ some edge of $K$, and the numerical flux $f_{K,e}:~\RR^2 \to \RR$ is chosen to satisfy the properties of consistency, conservation and monotonicity, that is, in our case
\begin{subequations}
\bei

\item \text{Consistency property:}
\bel{eq:consis22} 
f_{K, e^0}(v, v)
=
f(v),
\qquad
v \in \RR.
\ee
\item \text{Conservation property:}
\bel{eq:conser22} 
f_{K, e^0}(u, v) 
=
f_{K_{e^0}, e^0}(v, u),  
\qquad
u, v \in \RR.
\ee
\item \text{Monotonicity property:}
\bel{eq:monoto22} 
\omega_{K, e^0} \del_u f_{K, e^0}(u, v)
\geq 
0,
\qquad
\omega_{K, e^0} \del_v f_{K, e^0}(u, v)
\leq 
0,
\qquad
u, v \in \RR.
\ee
\eei
\end{subequations}
%


The finite volume approximations are then given by the explicit scheme  
\bel{eq:fv}
v^{n+1}_K
= 
v^n_K -{\tau \over |K|}\sum_{e \in \del K} |e|f_{K, e}(v^n_K, v^n_{K_e}) \omega_{K, e}
+ {\tau \over |K|} \sum_{e\in \del K} |e| f(v^n_K) \omega_{K, e} 
+ \big( f(v^n_K) + h(v^n_K) \big) \tau \theta_K.
\ee
For the sake of stability, we impose the CFL stability condition  
\bel{eq:CFL1}
{\tau p_K \over |K|} \max_{\Tcal_h, \del K} \sup_{-1 \leq u, v\leq 1 \atop u \neq v} {{{f_{e,K}(u,v)- f_{e,K}(v,v)}\over{u- v}}} \omega_{K, e} 
\leq {1\over 2},
\ee
as well as the source stability condition 
\bel{eq:CFL2}  
\tau \max_{\Tcal_h} \theta_K  \max_{-1\leq u \leq 1}\big  (| f'(u) + h'(u) | \big) 
< {1 \over 2}. 
\ee 

Now we are ready to state our convergence result.  

\begin{theorem}
\label{thm:main}
Consider the Cauchy problem for the balance law \eqref{eq:Nmodel} posed on the domain $\Omega$
under the assumption \eqref{atr:assumption1} and \eqref{atr:assumption2}. Impose the initial condition \eqref{eq:initial} with $v: [0, +\infty) \times \Omega \to [-1, 1]$ in $L^1(\Omega)$. Let $\Tcal_h$ be a triangulation and $\tau=\tau(h)$ be the time increment, satisfying 
\bel{eq:time-step}
\tau \to 0, \qquad {h^2 \over  \tau} \to 0, \qquad as~ h \to 0.
\ee 
Let $f_{K, e}$ be a family of numerical flux satisfying the consistency, conservation and monotonicity conditions in \eqref{eq:consis22}--\eqref{eq:monoto22} 
and satisfies the CFL condition \eqref{eq:CFL1} and the stability condition \eqref{eq:CFL2}. 
Then the discrete scheme \eqref{eq:fv} uniquely defines the family of approximate solution $v^n_K$. 
By defining a piecewise constant function $v^h : \RR_+ \times \Omega \to \RR$ by
\bel{eq:piecewise-sol}
v_h(t, x) :=v^n_K, \quad n\tau\leq t< (n+1)\tau,
\quad x\in K,
\ee
then the sequence $v_h: [0, +\infty) \times \Omega \to [-1, 1]$ is uniformly bounded in 
$L_\loc^\infty\big( [0, +\infty), L^1(\Omega)\big)$ 
and converges almost everywhere to an entropy solution $v: [0, +\infty) \times \Omega \to [-1, 1]$ (in the sense of Definition \ref{def-entropy})
$v \in L_\loc^\infty\big([0, +\infty), L^1(\Omega) \big)$. 
\end{theorem}

The above theorem implies the existence and stability of weak solutions for our model.

\begin{corollary} 
Consider the Cauchy problem for the balance law \eqref{eq:Nmodel} posed on the domain $\Omega$
under the assumption \eqref{atr:assumption1} and \eqref{atr:assumption2}. Impose the initial condition \eqref{eq:initial} with $v_0 \in L^1(\Omega)$ and $\| v_0 \|_{L^\infty(\Omega)} \leq 1$.
Then there exists an entropy solution $v: [0, +\infty) \times \Omega \to [-1, 1]$ in $L_\loc^\infty\big([0, +\infty), L^1(\Omega) \big)$ to this problem. 
\end{corollary}


\subsection{Discrete entropy inequalities}

Entropy inequalities play a key role in the proof of Theorem \ref{thm:main}.

\begin{proposition}
\label{p:maximum}
Under the assumptions in Theorem~\ref{thm:main}, the finite volume approximations satisfy the discrete maximum principle:
\bel{eq:maximum}
\max_{\Tcal_h}|v^n_K|
\leq 1.
\ee
\end{proposition}

\begin{proof}
We first assume $\max_{\Tcal_h}|v^n_K|\leq 1$ for all elements $K\in \Tcal_h$. We rewrite \eqref{eq:fv} as
$$
\aligned
v^{n+1}_K=&\Big({1+ {\tau \over |K|}\sum_{e\in \del K}{{f_{K,e}(v^n_K, v^n_{K_e})- f_{K,e}(v^n_K, v^n_K})\over{v^n_{K_e}- v^n_K}}|e| \omega_{K, e}}\Big)v^n_K
\\
                   &+{\tau \over |K|}\sum_{e\in \del K}{-{{f_{K,e}(v^n_K, v^n_{K_e})- f_{K,e}(v^n_K, v^n_K)}\over{v^n_{K_e}- v^n_K}}}|e| \omega_{K, e} v^n_{K_e}
                   \\
                   &-{\tau \over |K|}\sum_{e \in \del K} |e| \omega_{K, e} f_{K, e}(v^n_K, v^n_K) + {\tau \over |K|} \sum_{e\in \del K} |e| f(v^n_K) \omega_{K, e}
                   \\
                   & + \tau \big( f(v^n_K) + h(v^n_K) \big)\theta_K
\endaligned
$$
which with an obvious notation we express in the form
\bel{eq:301}
v^{n+1}_K 
 = A^n_K v^n_K+ \sum_{e \in \del K} A^n_{K, e} v^n_{K_e} + B_K \big( f(v^n_K) + h(v^n_K) \big). 
\ee
We have observed here that $A^n_K+ \sum_{e \in \del K} A^n_{K, e}=1$.

The monotonicity of $f_{K,e}$ implies 
$$
A^n_{K, e} = 
-{\tau \over |K|}{{f_{K,e}(v^n_K, v^n_{K_e})- f_{K,e}(v^n_K, v^n_K)}\over{v^n_{K_e}- v^n_K}}|e| \omega_{K, e} 
\geq 0,
$$
while the CFL condition \eqref{eq:CFL1} gives us 
$$
\aligned
 \sum_{e \in \del K} A^n_{K, e} 
=
-{\tau \over |K|} \sum_{e\in \del K}{{f_{K,e}(v^n_K, v^n_{K_e})- f_{K,e}(v^n_K, v^n_K}) \over {v^n_{K_e}- v^n_K}}|e| \omega_{K, e} 
\leq{1\over 2}. 
\endaligned
$$
Therefore, we have  $A^n_K \geq {1\over 2}$. 

On the other hand, since $f(1) + h(1) = 0$ and $v^n_{K_e} \leq 1$ we have
$$
\aligned
v^{n+1}_K\leq &  A^n_K v^n_K+ \sum_{e \in \del K} A^n_{K, e} + B_K \Big(f(v^n_K) + h(v^n_K) - f(1) - h(1) \Big)
\\
                \leq &  A^n_K v^n_K + (1- A^n_K)- B_K \max_{-1\leq u \leq 1}|f'(u) + h'(u)| (v^n_K - 1)
                \\
                \leq & 1,
\endaligned
$$
where we used the source stability condition \eqref{eq:CFL2}. 
Similarly, we find 
$$
\aligned
v^{n+1}_K\geq & A^n_K v^n_K- \sum_{e \in \del K} A^n_{K, e} + B_K \Big( f(v^n_K) + h(v^n_K) - f(-1) - h(-1) \Big)
\\
                \geq & A^n_K v^n_K - (1- A^n_K)- B_K \max_{-1\leq u \leq 1}|f'(u) + h'(u)| (v^n_K + 1)
                \\
                \geq & -1.  \hskip5.cm \qedhere
\endaligned 
$$ 
\end{proof}


Now, we state a convex decomposition of $v^{n+1}_K$, which plays an important role in deriving the discrete entropy inequalities given below. For each $K$ and $e$, we define
\begin{subequations}
\bel{eq:conv-decom1}
\widetilde{v}^{n+1}_{K, e}
:= 
v^n_K- {\tau p_K \omega_{K, e} \over |K|} \Big( f_{K,e}(v^n_K, v^n_{K_e})- f_{K,e}(v^n_K, v^n_K) \Big),
\ee
and 
\bel{eq:conv-decom2}
v^{n+1}_{K, e}
:= 
\widetilde{v}^{n+1}_{K, e}- {\tau \over |K|}\sum_{e \in \del K} \omega_{K, e} |e| f_{K, e}(v^n_K, v^n_K)
+ {\tau p_K \over |K|} f(v^n_K) \omega_{K, e} + \tau \theta_K \Big( f(v^n_K) + h(v^n_K) \Big).
\ee
In view of \eqref{eq:fv} and the consistency property of $f_{K, e}$, we have
\bel{eq:conv-decom3}
v^{n+1}_K
=
{1 \over p_K}\sum_{e \in \del K} |e| v^{n+1}_{K, e}.
\ee
\end{subequations}

The following lemma provides a standard result concerning the existence of discrete entropy flux terms and an entropy inequality relating $\widetilde{v}^{n+1}_{K, e}$ and $v^{n+1}_K$. We omit the proof and refer to \cite{ALO} and the references therein.

\begin{lemma}[Discrete entropy inequalities] 
Let $(U, F)$ be a convex entropy pair. Then there exists a family of discrete entropy flux functions $F_{K, e}: \RR^2 \to \RR$ satisfying the following conditions:  
\begin{subequations}
\bei

\item Consistency with the entropy flux $F$:
         \bel{eq:consistF}
         F_{K, e}(u, u) = F(u), \qquad u\in \RR.
\ee

\item Conservation property:
         \bel{eq:conservF}
         F_{K, e}(u, w)  = F_{K_e, e}(w, u), 
\qquad u, w\in \RR.
\ee

\item Discrete entropy inequality:
\bel{eq:entropyF}
\aligned
&         U(\widetilde{v}^{n+1}_{K, e})- U(v^n_K) + {\tau p_K \omega_{K, e} \over |K|}\Big( F_{K,e}(v^n_K, v^n_{K_e})- F_{K,e}(v^n_K, v^n_K) \Big)
\\
&         \leq 
         {\tau \theta_K} \Big( f(v^n_K) + h(v^n_K) \Big) U'(v^n_K).
\endaligned
\ee
\eei 
\end{subequations}
\end{lemma}

Equivalently, \eqref{eq:entropyF} can be written in terms of $v^{n+1}_{K,e}$ and $v^n_K$ as
\bel{eq:311}
\aligned
& U(v^{n+1}_{K, e})- U(v^n_K) + {\tau p_K \omega_{K, e} \over |K|} \Big( F_{K,e}(v^n_K, v^n_{K_e})- F_{K,e}(v^n_K, v^n_K) \Big)
\\
& \leq  {\tau \theta_K} \Big( f(v^n_K) + h(v^n_K) \Big) U'(v^n_K) + R^{n+1}_{K, e},
\endaligned
\ee
with $R^{n+1}_{K, e}:=U(v^{n+1}_{K, e})- U(\widetilde{v}^{n+1}_{K, e})$.
The entropy dissipation estimate below will serve to establish the convergence result.

\begin{proposition}[Discrete entropy balance law] 
Let $U: \RR \to \RR$ be a strictly convex function and set $\alpha:=\inf_{v \in [-1,1]} U''(v)$. Then for all $n$ one has 
\bel{eq:312}
\aligned
& \sum_{K \in \Tcal_h}|K|U(v^{n+1}_K) + {\alpha \over 2}\sum_{K \in \Tcal_h, e \in \del K}{{|e||K|}\over{p_K}}|v^{n+1}_{K,e}-v^{n+1}_K|^2
\\
& \leq 
\sum_{K \in \Tcal_h}|K|U(v^n_K) + \sum_{K \in \Tcal_h, e \in \del K}\tau |e| \omega_{K, e} F_{K,e}(v^n_K, v^n_K) + \sum_{K \in \Tcal_h} {\tau |K| \theta_K} \Big( f(v^n_K) + h(v^n_K) \Big) U'(v^n_K)
\\
& \quad
+ \sum_{K \in \Tcal_h, e \in \del K}{|e||K| \over p_K}R^{n+1}_{K, e}.
\endaligned
\ee
\end{proposition}

\begin{proof}
By Lemma 3.5 in \cite{CCL2} and the convex decomposition identity \eqref{eq:conv-decom3}, we have
$$
\sum_{K \in \Tcal_h}|K|U(v^{n+1}_K) + {\alpha \over 2}\sum_{K \in \Tcal_h, e \in \del K}{{|e||K|}\over{p_K}}|v^{n+1}_{K,e}-v^{n+1}_K|^2
\leq
\sum_{K \in \Tcal_h, e \in \del K}{{|e||K|}\over{p_K}}U(v^{n+1}_{K,e}).
$$
Next, we multiply by $|e||K|/p_K$ in \eqref{eq:311}, and sum up over all $K$ and $e$,
$$
\aligned
& \sum_{K \in \Tcal_h, e \in \del K}{{|e||K|}\over{p_K}}U(v^{n+1}_{K,e})- \sum_{K \in \Tcal_h}|K|U(v^n_K)- \sum_{K \in \Tcal_h, e \in \del K}{\tau  |e| \omega_{K, e}}F_{K,e}(v^n_K, v^n_K))
\\
& \leq 
\sum_{K \in \Tcal_h} {\tau |K| \theta_K} \Big( f(v^n_K) + h(v^n_K) \Big) U'(v^n_K) + \sum_{K \in \Tcal_h, e \in \del K}{|e||K| \over p_K}R^{n+1}_{K, e}.
\endaligned
$$
This leads us to \eqref{eq:312}.
\end{proof}

For the proofs of the following lemmas, see \cite{ALO} and the references therein for details.
The local entropy inequalities read as follows. 

\begin{lemma} One has
$$
\aligned
& 
{|K|\over p_K}U(v^{n+1}_{K, e})- {|K|\over p_K}U(v^n_K) + {|K_e|\over p_{K_e}}U(v^{n+1}_{K_e, e})- {|K_e|\over p_{K_e}}U(v^n_{K_e}) 
\\
& + \tau \Big( F(v^n_{K_e}) \omega_{K_e, e} - F(v^n_K) \omega_{K, e} \Big)
\\
&\leq 
{\tau |K| \theta_K \over p_K} \Big( f(v^n_K) + h(v^n_K) \Big) U'(v^n_K) + {\tau |K_e| \theta_{K_e} \over p_{K_e}} \Big( f(v^n_{K_e}) + h(v^n_{K_e}) \Big) U'(v^n_{K_e})
\\
& \quad + {|K|\over p_K}R^{n+1}_{K,e}
+ {|K_e|\over p_{K_e}}R^{n+1}_{K_e,e}.
\endaligned
$$
\end{lemma}

The global entropy inequalities read as follows. 

\begin{lemma}
Let $(U, F)$ be a convex entropy pair and let $\phi= \phi(t, x)\in C_c ([0, T)\times \Omega)$ be a test function. For each element $K$ and each face $e\in \del K$, set
\begin{subequations}
\be 
\phi^n_e
:= 
{1 \over {\tau |e|}}\int_{t_n}^{t_{n+1}} \int_e \phi(t, x) \, dSdt, \quad \quad \widehat{\phi}^n_K:=\sum_{e\in \del K}{|e| \over p_K}\phi^n_e,
\ee
and
 \be 
\widehat{\del_t \phi}_K^n
:= 
{1 \over \tau}(\widehat{\phi}^n_K- \widehat{\phi}^{n-1}_K).
\ee
\end{subequations}
Then one has 
\bel{eq:331}
\aligned
& \sum_{n=1}^\infty \sum_{K \in \Tcal_h}\int_{t_n}^{t_{n+1}}\int_K 
                                                     \Bigg( U(v^n_K)\widehat{\del_t \phi}_K^n+ F(v^n_K) \del_j \Big( \big(1 - 2 M / r \big)(x^j/r)\phi(t, x) \Big)
                                                     \\
                                        &\hskip3cm    + (2 M / r_K^2) U'(v^n_K) \Big( f(v^n_K) + h(v^n_K) \Big) \widehat{\phi}_K^n \Bigg) \, dxdt
+ \sum_{K \in \Tcal_h}\int_K U(v^0_K)\widehat{\phi}^0_K \, dx 
\\
&\geq 
\sum^{+ \infty}_{n=0}\sum_{K \in \Tcal_h \atop e \in \del K} 
\hskip-.2cm \Bigg( {|K| |e| \over p_K}\phi^n_e R^{n+1}_{K, e}
+
 \int_{t_n}^{t_{n+1}}\int_e F(v_K^n) \big( \phi_e^n - \phi(t, x) \big) \Big(1 - 2 M / r\Big) \big(x / r) \cdot n_{K, e} \, dSdt \Bigg).
\endaligned
\ee 
\end{lemma}


\subsection{Measure-valued solutions and strong convergence}

We are now in a position to  complete our proof of Theorem \ref{thm:main}. Based on the entropy inequalities we have established, we are able to pass the limit in the inequality \eqref{eq:331} as $h\to 0$. Then we associate with a subsequence of $v_h$ (which is uniformly bounded in $[0, T) \times \Omega \to \RR$ for fixed $T$) a Young measure $\nu : [0, T) \times \Omega \to \Prob(\RR)$, which is a family of probability measures in $\RR$ parametrized by $(t, x)\in [0, T) \times \Omega$. We then show that the Young measure, describing all the weak-star limits of $v_h$, is an entropy measure-valued solution in the sense of DiPerna. The strong convergence result follows from the DiPerna's uniqueness theorem, see \cite{DiPerna}.

The Young measure allows us to write, for every continuous function $a : \RR \to \RR$, 
\bel{eq:Young-m}
a(v_h) \rightharpoonup \langle \nu, a \rangle \quad \text{ as } h \to 0,
\ee
in the $L^\infty$ weak-star topology. As presented in \cite{ALO}, it suffices to show that $\nu$ is an entropy measure-valued solution to our balance law, in order to imply that $\nu_{t, x}$ reduce to a Dirac mass $\delta_{v(t, x)}$ if this is true at the initial time $t=0$. The convergence in \eqref{eq:Young-m} then holds in a strong sense and $v_h$ converges to the entropy solution $v$ to the Cauchy problem.

\begin{lemma}
\label{lem:entropy-int-m}
Let $\nu: [0, T) \times \Omega \to \Prob(\RR)$ be the Young measure associated with the sequence $v_h$. Then for every convex entropy pair $(U, F)$ one has
\bel{eq:entropy-m}
\aligned
0 
&\leq
\int_{[0, T)}\int_\Omega  \Bigg( \langle\nu_{t, x}, U \rangle \del_t \phi(t, x) + \langle \nu_{t, x}, F \rangle \del_j \Big( \big( 1 - {2M \over r} \big) {x^j \over r} \phi(t, x)  \Big) 
\\
& \hskip2.cm 
+ \langle \nu_{t, x}, U' \Big( f + h \Big) \rangle {2M \over r^2} \phi(t, x) \Bigg) \, dxdt 
+ \int_\Omega  U\big( v_0(x) \big)\phi(0, x)
\endaligned
\ee
for all non-negative test functions $\phi: [0, T) \times \Omega \to \RR_+$.
\end{lemma}
 
For all convex entropy pairs, we thus have
\be 
\aligned
\del_t \langle\nu, U \rangle
+ \big( 1 - {2M \over r} \big) {x^j \over r} \del_j \langle \nu, F \rangle - {2M \over r^2} \langle \nu, U' \big( f + h\big) \rangle 
\leq 0, 
\endaligned
\ee 
and the proof of of Theorem \ref{thm:main} is completed. 


\section*{Acknowledgments} The authors were supported by the Innovative Training Network (ITN) entitled  ModCompShock under the grant 642768.



\begin{thebibliography}{10}  

\bibitem{ALO} \auth{P. Amorim, P.G. LeFloch, and B. Okutmustur}, 
Finite volume schemes on Lorentzian manifolds, 
Comm. Math. Sc. 6 (2008), 1059--1086. 

\bibitem{BL} \auth{Y. Bakhtin and P.G. LeFloch}, 
Ergodicity of spherically symmetric fluid flows outside of a Schwarzschild black hole with random boundary forcing, 
Stoch PDE: Anal. Comp. 6 (2018), 746--785 . 

\bibitem{CLO} \auth{T. Ceylan, P.G. LeFloch, and B. Okutmustur,}
A finite volume method for the relativistic Burgers equation on a FLRW background spacetime,  
{Commun. Comput. Phys.} 23 (2018), 500--519.  

\bibitem{CCL2} \auth{B. Cockburn, F. Coquel, and P.G. LeFloch},
Convergence of finite volume methods for multi-dimensional conservation laws,
SIAM J. Numer. Anal. 32 (1995), 687--705. 

\bibitem{DB}\auth{K.A. Dennison and T.W. Baumgarte,}  
A simple family of analytical trumpet slices of the Schwarzschild spacetime, 
Class. Quant. Grav. 31 (2014), 117001.

\bibitem{DiPerna} \auth{R. J. DiPerna,}
Measure-valued solutions to conservation laws,
Arch. Rational Mech. Anal. 88 (1985), 223--270. 

\bibitem{GS} \auth{J. Giesselman and P.G. LeFloch,} 
Formulation and convergence of the finite volume method for conservation laws on spacetimes with boundary, Preprint ArXiv:1607.03944,

\bibitem{KMS15} \auth{D. Kr\"oner, T. M\"uller, and L. M. Strehlau,}
Traces for functions of bounded variation on manifolds with applications to conservation laws on manifolds with boundary, 
SIAM J. Math. Anal. 47 (2015), 3944--3962.

\bibitem{Kruzkov} \auth{S. Kruzkov,}
First-order quasilinear equations with several space variables,
Math. USSR Sb. 10 (1970), 217--243.  
 
\bibitem{LM} \auth{P. G. LeFloch and  H. Makhlof,}
A geometry-preserving finite volume scheme for compressible fluids on Schwarzschild spacetime, 
Commun. Comput. Phys. 15 (2014), 827--852.
 
\bibitem{LMO} \auth{P. G. LeFloch,  H. Makhlof, and B. Okutmustur,}
Relativistic Burgers equations on curved spacetimes. Derivation and finite volume approximation, 
SIAM J. Numer. Anal.  50 (2012),  2136--2158. 

\bibitem{LO}  \auth{P. G. LeFloch and B. Okutmustur,} 
Hyperbolic conservation laws on spacetimes. A finite volume scheme based on differential forms, 
Far East J. Math. Sci. 31 (2008), 49--83.

\bibitem{LX1} \auth{P.G. LeFloch and S. Xiang,}
Weakly regular fluid flows with bounded variation
on the domain of outer communication of a Schwarzschild black hole spacetime, 
J. Math. Pures Appl. 106 (2016), 1038--1090.   

\bibitem{LX2} \auth{P.G. LeFloch and S. Xiang,}
Weakly regular fluid flows with bounded variation on the domain of outer communication of a Schwarzschild black hole spacetime. II, 
J. Math. Pure Appl. 122 (2019), 272--317. 

\end{thebibliography}
\end{document}